\documentclass[11pt,leqno]{article}
\usepackage{graphicx}
\usepackage{amsfonts, amsthm}
\usepackage{amsxtra}
\usepackage{tikz}
\usepackage[latin1]{inputenc}
\usepackage{fontenc}
\usepackage[dvips]{epsfig}
\begin{document}
\newtheorem{thm}{Theorem}[section]
\newtheorem{cor}[thm]{Corollary}
\newtheorem{lem}[thm]{Lemma}
\newtheorem{exm}[thm]{Example}
\newtheorem{prop}[thm]{Proposition}
\theoremstyle{definition}
\newtheorem{defn}[thm]{Definition}
\theoremstyle{remark}
\newtheorem{rem}[thm]{Remark}
\numberwithin{equation}{section}
\begin{center}
{\LARGE\bf On  multidimensional fixed-point theorems and their applications}
\footnote{MSC2000: 54H25, 47H10.
Keywords and phrases:  fixed-point,  partially ordered metric space, Hammerstein integral equation}
\\
\vspace{.25in}{\large{H. Akhadkulov\footnote{School of Quantitative Sciences, University Utara Malaysia, CAS 06010, UUM  Sintok, Kedah Darul Aman, Malaysia.}}}$^{,*}$,
S. Akhatkulov\footnote{Faculty of Applied Mathematics and Informatics, Samarkand State University, Boulevard st. 15,140104 Samarkand, Uzbekistan.

* Corresponding author: E-mail: habibulla@uum.edu.my }, T. Y. Ying$^2$ and R. Tilavov$^3$

\end{center}

\title{On  multidimensional fixed-point theorems and their applications}

\begin{abstract}
The purpose of this paper is to present 
some multidimensional fixed-point theorems and their applications.
For this, we provide a multidimensional fixed point theorem and then using this  theorem
we prove the existence and uniqueness of a solution of a nonlinear Hammerstein integral equation.
Moreover, we provide an example to illustrate the hypotheses and the abstract
result of this paper.
\end{abstract}
 \section{Introduction and Preliminaries}
Many problems which arise in mathematical physics, engineering,
biology, economics and etc., lead to mathematical models described by nonlinear
integral equations. For instance, the Hammerstein integral
equations appear in nonlinear physical phenomena such as electro-magnetic fluid
dynamics, reformulation of boundary value problems with a nonlinear boundary
condition (see \cite{atk}). A Hammerstein integral equation is introduced as follows
$$
x(t)=\int_{a}^{b}\mathcal{G}(t,s)H(s,x(s))ds+p(t).
$$
The aim of this paper is to investigate this integral equation
under a certain conditions of $\mathcal{G}$ and $H.$
For this, we use the methods of  multidimensional fixed point theorems.
The concept of multidimensional fixed point i.e., \emph{$\Upsilon$-fixed point} was introduced
by Rold\`{a}n  \emph{et. al.} \cite{Roldan et al 2012, Roldan2014}.
This notion covers the concepts of \emph{coupled, tripled, quadruple} fixed point.
We refer the reader to the references  \cite{akhad5, Abbas1, Berinde Borcur, Guo Laksh, Habib} in which were introduced the concept of coupled, tripled, quadruple fixed points and obtained related theorems.
The uniqueness and  existence theorems of multidimensional fixed point and their
applications  to nonlinear integral equations, matrix equations and the system of matrix equations
have been developed in    \cite{akhad1}-\cite{akhad4}, \cite{akhad6}.
In this paper, by using  multidimensional fixed point theorems,  we prove the existence and uniqueness of solution of a nonlinear Hammerstein integral equation under a certain conditions of $\mathcal{G}$ and $H.$
Moreover, we provide an example to illustrate the hypotheses and the abstract
result of this paper. Let us introduce some necessary concepts and tools which help us to formulate
our theorems.  Denote by $(X,d,\preceq)$ a \emph{partially
ordered metric space}.

\begin{defn}
An ordered metric space $(X,d,\preceq)$ is called \emph{regular}  if it satisfies  the following:
\begin{enumerate}
  \item [-] if $\{x_{m}\}$ is a nondecreasing sequence and $\{x_{m}\}\overset{d}{\rightarrow} x$,
then $x_{m}\preceq x$ for all $m;$
  \item [-] if $\{y_{m}\}$ is a nonincreasing sequence and $\{y_{m}\}\overset{d}{\rightarrow} y$
then $y_{m}\succeq y$ for all $m$.
\end{enumerate}
\end{defn}
Taking a natural number $k\geq 2$ we consider the set $\Lambda_{k}=\{1,2,\ldots,k\}.$
Let $\{\mathcal{A},\mathcal{B}\}$ be a partition  of $\Lambda_{k}$
that is $\mathcal{A}\cup \mathcal{B}=\Lambda_{k}$ and $\mathcal{A}\cap \mathcal{B}=\emptyset.$
Using this partition and  partially ordered metric space $(X, d, \preceq)$
we define a $k$-dimensional partially ordered metric space $(X^{k},\mathbf{\textbf{d}}_{k}, \preceq_{k})$
as follows:
\begin{itemize}
  \item the $k$-cartesian power of a set $X$
   $$
   X^{k}=\underset{k}{\underbrace{X\times X\times\cdot\cdot\cdot \times X}}=\{(\textbf{x}=(x_{1}, x_{2},...,x_{k})):| x_{i}\in X
   \,\,\,\,\text{for all}\,\,\, i\in \Lambda_{k}\};
   $$
  \item the maximum metric $\mathbf{\textbf{d}}_{k}:X^{k}\times X^{k}\rightarrow [0, +\infty),$ given by
$$
\mathbf{\textbf{d}}_{k}(\textbf{x},\textbf{y})=\underset{1\leq i\leq k}{\max}\{d(x_{i},y_{i})\},
$$
where $\textbf{x}=(x_{1}, x_{2},...,x_{k}), \textbf{y}=(y_{1}, y_{2},...,y_{k})\in X^{k};$
  \item the partial order w.r.t $\{\mathcal{A},\mathcal{B}\}$ that is, for any
  $\textbf{x}=(x_{1}, x_{2},...,x_{k})$ and $\textbf{y}=(y_{1}, y_{2},...,y_{k}) \in X^{k}$ we have
\begin{equation*}\label{n1}
 \textbf{x}\preceq_{k} \textbf{y} \Leftrightarrow \left\{
\begin{array}{ll}
x_{i}\preceq y_{i}, & \text{if }\quad i\in \mathcal{A},\\
x_{i}\succeq y_{i}, & \text{if} \quad i\in \mathcal{B}.
\end{array}
\right.
\end{equation*}
\end{itemize}
 It is easy to see that if  $(X, d)$ is a complete metric space,
 then  $(X^{k},\mathbf{\textbf{d}}_{k})$ is a complete metric space.
\begin{defn}
We say that a mapping  $F:X^{k}\rightarrow X$ has the \emph{mixed monotone} property w.r.t partition  $\{\mathcal{A}, \mathcal{B}\},$
if $F$ is  monotone nondecreasing in arguments of $\mathcal{A}$ and monotone nonincreasing in arguments of $\mathcal{B}.$
\end{defn}
We define  the following set of mappings:
$$
 \Omega_{\mathcal{A},\mathcal{B}} =\{\sigma:\Lambda_{k}\rightarrow \Lambda_{k}:\sigma(\mathcal{A})
 \subseteq \mathcal{A},\,\, \sigma(\mathcal{B})\subseteq \mathcal{B}\},
  $$
  $$
  \Omega'_{\mathcal{A}, \mathcal{B}} =\{\sigma:\Lambda_{k}\rightarrow \Lambda_{k}:\sigma(\mathcal{A})\subseteq \mathcal{B},\,\, \sigma(\mathcal{B})\subseteq \mathcal{A}\}.
$$
Let $\Upsilon=(\sigma_{1},\sigma_{2},\ldots,\sigma_{k})$ be $k$-tuple of mappings
of $\sigma_{i}:\Lambda_{k}\rightarrow \Lambda_{k}$
such that $\sigma_{i}\in \Omega_{\mathcal{A},\mathcal{B}}$ if $i\in \mathcal{A}$
and $\sigma_{i}\in \Omega'_{\mathcal{A},\mathcal{B}}$ if $i\in \mathcal{B}.$
In the sequel we consider only such kind of $k$-tuple of mappings.
\begin{defn}
A point $\textbf{x}=(x_{1},x_{2},\ldots,x_{k})\in X^{k}$ is called
\emph{$\Upsilon$-fixed  point} of a mapping $F:X^{k}\rightarrow X$ if
$$
F(x_{\sigma_{i} (1)},x_{\sigma_{i} (2)},\ldots,x_{\sigma_{i} (k)})=x_{i}
$$
for all $i\in \Lambda_{k}.$
\end{defn}
\section{A Multidimensional fixed point theorem}
In this section we provide a multidimensional fixed point theorem
which will be used in the next section. We need the following definition.
\begin{defn}
A function $\psi:[0, +\infty)\rightarrow [0, +\infty)$ is called \emph{altering distance function}, if
$\psi$ is continuous, monotonically increasing and $\psi(\{0\})=\{0\}.$
\end{defn}
The following theorem has been obtained by Akhadkulov et. al in \cite{akhad1}.
\begin{thm}\label{roldan2}
Let $(X,d,\preceq)$ be a complete partially ordered metric space.
Let $\Upsilon:\Lambda_{k}\rightarrow \Lambda_{k}$ be a $k$-tuple  mapping  $\Upsilon=(\sigma_{1}, \sigma_{2},...,\sigma_{k})$
 such that $\sigma_{i}\in \Omega_{\mathcal{A}, \mathcal{B}}$ if $i\in \mathcal{A}$  and
  $\sigma_{i}\in \Omega'_{\mathcal{A}, \mathcal{B}}$ if $i\in \mathcal{B}.$
Let $F:X^{k}\rightarrow X$ be a mapping which obeys the following conditions:
\begin{enumerate}
  \item[(i)] there exists an altering distance function $\psi,$
  an upper semi-continuous function $\theta:[0, +\infty)\rightarrow [0,+\infty)$
   and a  lower semi-continuous function  $\varphi:[0,+\infty)\rightarrow [0,+\infty)$
   such that for all
   $\mathbf{x}=(x_{1},x_{2},\ldots,x_{k}), \mathbf{y}=(y_{1},y_{2},\ldots,y_{k}),$
   $\mathbf{x},\mathbf{y}$  with $\mathbf{x}\preceq_{k}\mathbf{y}$ we have
\[
\psi(d(F(\mathbf{x}),F(\mathbf{y})))\leq
\theta(\mathbf{\textbf{d}}_{k}(\mathbf{x}, \mathbf{y}))-\varphi( \mathbf{\textbf{d}}_{k}(\mathbf{x}, \mathbf{y}))
\]
where $\theta(0)=\varphi(0)=0$ and $\psi(x)-\theta(x)+\varphi(x)>0$ for all $x>0$;
  \item[(ii)] there exists $\mathbf{x}^{0}=(x^{0}_{1},x^{0}_{2},\ldots,x^{0}_{k})$ such that
  $x^{0}_{i}\preceq_{i} F(x^{0}_{\sigma_{i} (1)},x^{0}_{\sigma_{i} (2)},\ldots,x^{0}_{\sigma_{i} (k)})$ for all $i \in \Lambda_{k};$

  \item[(iii)] $F$ has the mixed monotone property w.r.t $\{\mathcal{A}, \mathcal{B}\}$;

  \item[(iv)](a) $F$ is continuous or\\
   (b) $(X,d,\preceq)$ is regular.
\end{enumerate}
Then $F$ has a  $\Upsilon$-fixed point.
Moreover
\begin{itemize}
  \item [(v)] if for any $\mathbf{x}=(x_{1},x_{2},\ldots,x_{k}), \mathbf{y}=(y_{1},y_{2},\ldots,y_{k})$
    there exists a point
$\mathbf{z}=(z_{1},z_{2},\ldots,z_{k})$ such that
$\mathbf{x}\preceq_{k}\mathbf{z}$ and $\mathbf{y}\preceq_{k}\mathbf{z},$ then $F$ has a unique $\Upsilon$-fixed point  $\mathbf{x}^{*}=(x^{*}_{1},x^{*}_{2},\ldots ,x^{*}_{k}).$
\end{itemize}
\end{thm}

\section{An application of Theorem \ref{roldan2}}
In this section, we apply  Theorem \ref{roldan2} to a
nonlinear Hammerstein integral equation  to
show the existence and  uniqueness of solution.
Let $T>1$ be a real number.
 Consider the following nonlinear  Hammerstein  integral equation on $C([1,T])$:
\begin{equation}\label{eq41}
x(t)=\int_{1}^{T} \mathcal{G}(t,s)\Big[{\sum_{i=1}^{2m} f_{i}(s,x(s))}\Big]ds+p(t), \quad t\in [1,T].
\end{equation}
In order to show the existence of a solution of equation (\ref{eq41}) we assume:
\begin{enumerate}
  \item[(a)] $f_{i}:[1,T]\times \mathbb{R}\rightarrow \mathbb{R},\quad 1\leq i \leq 2m$ are continuous;
  \item[(b)] $p:[1,T]\rightarrow \mathbb{R}$ is continuous;
  \item[(c)] $\mathcal{G}:[1,T]\times [1,T]\rightarrow [0,\infty)$ is continuous;
  \item[(d)]there exist positive constants $\eta_{1}, \eta_{2},...,\eta_{2m}$  such that
  $$
\max_{1\leq i \leq 2m}\eta_{i}\leq \Big(2m\underset{0\leq t\leq T}{\max}\int^{T}_{1}\mathcal{G}(t,s)ds\Big)^{-1}
  $$
 for all $1\leq i \leq 2m$ and
    $$
    0\leq f_{2i-1}(s,y)-f_{2i-1}(s,x)\leq\eta_{2i-1}\log\Big(1+y-x\Big),
    $$
    $$
    -\eta_{2i}\log\Big(1+y-x\Big)\leq f_{2i}(s,y)-f_{2i}(s,x)\leq 0
    $$
    for all $x,y\in\mathbb{R}, \,\,\,  y\geq x$ and  $1\leq i\leq m.$
  \item[(e)] there exist continuous functions $y^{0}_{1},y^{0}_{2},\ldots,y^{0}_{2m}:[1,T]\rightarrow \mathbb{R}$ such that
    $y^{0}_{2r-1}(t)\leq H_{2r-1}(t),$ $1\leq r\leq m$ and $y^{0}_{2r}(t)\geq H_{2r}(t),$ $1\leq r\leq m$
    for all $t\in [0,T]$  where
    $$
    H_{1}(t)=\int^{T}_{1}\mathcal{G}(t,s)\Big[\sum_{i=1}^{2m} f_{i}(s,y^{0}_{i}(s))\Big]ds+p(t)
    $$
   and
   $$
    H_{r}(t)=\int^{T}_{1}\mathcal{G}(t,s)\Big[\sum_{i=1}^{2m-r+1} f_{i}(s,y^{0}_{i+r-1}(s))+\sum_{\ell=0}^{r-2} f_{2m-\ell}(s,y^{0}_{r-1-\ell}(s))\Big]ds+p(t),
    $$
for $2\leq r\leq 2m.$
\end{enumerate}
Note that the equation (\ref{eq41}) has been studied in \cite{H1},
under the similar assumptions. The main difference  is the
contraction condition i.e., the assumption  (e).
We have the following.
\begin{thm}\label{th4}
Under assumptions (a)-(e), equation (\ref{eq41}) has a unique solution in $C[1,T].$
\end{thm}
\begin{proof}
The proof of this theorem is similar to the proof of the main theorem of \cite{H1}.
Therefore we give only the sketch of the proof.
  First, we define necessary notions as follow.
  Let  $X=C[1,T]$ be a space of continuous real functions defined on $[1,T]$
  endowed with the standard metric given by
$$
d(u,v)=\max_{1\leq t\leq T}\mid u(t)-v(t)\mid, \quad \text{for}\quad u,v\in X.
$$
A partial order $\preceq$ is defined  as follows: for any  $x,y\in C[1,T]$
we say
$$
x\preceq y\Leftrightarrow x(t)\leq y(t),\quad \text{for all}\quad t\in [1,T].
$$
Let $\Lambda_{2m}=\{1,2,\ldots,2m\}.$ Consider a partition
$$
\mathcal{A}=\{1,3,5,\ldots,2m-1\}\,\,\,\,
\text{and}\,\,\,\,\,
\mathcal{B}=\{2,4,6,\ldots,2m\}.
$$
 We choose  $\Upsilon=(\sigma_{1},\sigma_{2},\ldots,\sigma_{2m})$
 as follows:
\[
\Upsilon=\left( \begin{array}{cccc}
\sigma_{1}(1) & \sigma_{1}(2)& \ldots& \sigma_{1}(2m) \\
\sigma_{2}(1) & \sigma_{2}(2)& \ldots& \sigma_{2}(2m) \\
\sigma_{3}(1) & \sigma_{3}(2)& \ldots& \sigma_{3}(2m) \\
\ldots&\ldots&\ldots&\ldots\\
\sigma_{2m}(1) & \sigma_{2m}(2)& \ldots& \sigma_{2m}(2m)
 \end{array} \right)=
 \left( \begin{array}{cccccc}
1 & 2& \ldots& 2m-2&2m-1&2m \\
2 & 3& \ldots&2m-1&2m&1  \\
3 & 4& \ldots&2m&1&2\\
\ldots&\ldots&\ldots&\ldots&\ldots&\ldots\\
2m & 1& \ldots&2m-3&2m-2&2m-1
 \end{array} \right)
 \]
 Next we consider the operator $\mathbb{A}:X^{2m}\rightarrow X$
 $$
 \mathbb{A}(\mathbf{x})=\mathbb{A}(x_{1},x_{2},\ldots,x_{2m})=\int_{1}^{T} \mathcal{G}(t,s)\Big[{\sum_{i=1}^{2m} f_{i}(s,x_{i}(s))}\Big]ds+p(t),
 $$
where $t\in [1,T]$ and  $\mathbf{x}=(x_{1},x_{2},\ldots,x_{2m})\in X^{2m}$.
  Further, we show $\mathbb{A}$ satisfies all conditions of Theorem \ref{roldan2}.
 Let $\mathbf{x}=(x_{1},x_{2},\ldots,x_{2m}),\, \mathbf{z}=(z_{1},z_{2},\ldots,z_{2m})\in X^{2m}$.
  We define a metric in $X^{2m}$ as follows:
 $$
\mathbf{\textbf{d}}_{2m}(\mathbf{x},\mathbf{z})=\max_{ i\in \Lambda_{2m}}\{d(x_{i},z_{i})\}=\max_{i\in \Lambda_{2m}}\{\max_{1\leq t\leq T}\mid x_{i}(t)-z_{i}(t)\mid\}.
 $$
 \textbf{Step 1}.
We claim that the operator  $\mathbb{A}$ satisfies the first condition of Theorem \ref{roldan2} with
$$
  \psi(x)=x, \,\,\,\,\theta(x)=\log(1+x)\,\,\, \text{and}\,\,\,\,\varphi(x)=0.
$$
Indeed, from assumption $(\mathbf{d})$ it follows that
$$
\mathbb{A}(z_{1},z_{2},\ldots,z_{2m})(t)-\mathbb{A}(x_{1},x_{2},\ldots,x_{2m})(t)=
$$
$$
\int^{T}_{1}\mathcal{G}(t,s)\Big[\sum_{i=1}^{2m} f_{i}(s,z_{i}(s))- f_{i}(s,x_{i}(s))\Big]ds
\leq
$$
$$
2m(\max_{1\leq i\leq 2m}\eta_{i})\Big(\underset{1\leq t\leq T}{\max}\int^{T}_{1}\mathcal{G}(t,s)ds\Big)
 \cdot\log\Big(1+\mathbf{\textbf{d}}_{2m}(\mathbf{x},\mathbf{z})\Big)
$$
for any  $\mathbf{x}=(x_{1},x_{2},\ldots,x_{2m}), \mathbf{z}=(z_{1},z_{2},\ldots,z_{2m})\in X^{2m}$
 with $\mathbf{x}\preceq_{2m}\mathbf{z}.$
Hence
\begin{equation*}
d(\mathbb{A}(\mathbf{x}), \mathbb{A}(\mathbf{z}))\leq \log\Big(1+\mathbf{\textbf{d}}_{2m}(\mathbf{x},\mathbf{z})\Big)
\end{equation*}
that is
$$
\psi(d(\mathbb{A}(\mathbf{x}), \mathbb{A}(\mathbf{z})))\leq \theta({\mathbf{\textbf{d}}_{2m}(\mathbf{x},\mathbf{z})})-
\varphi({\mathbf{\textbf{d}}_{2m}(\mathbf{x},\mathbf{z})}).
$$
One can easily see  $\psi(x)-\theta(x)+\varphi(x)=x-\log(1+x)>0$ for all $x>0.$\\
\\
\textbf{Step 2}. There exists  $\mathbf{y}=(y^{0}_{1},y^{0}_{2},\ldots,y^{0}_{2m})\in X^{2m}$ such that the operator  $\mathbb{A}$ satisfies the second condition of Theorem \ref{roldan2}.
The proof of this claim follows from assumption $(\mathbf{e}).$\\
\\
\textbf{Step 3}. The operator $\mathbb{A}$ has mixed monotone property w.r.t $\{\mathcal{A},\mathcal{B}\}.$
The proof of this claim follows from assumptions $(\mathbf{c})$ and $(\mathbf{d}).$\\
\\
\textbf{Step 4}. We claim that $\mathbb{A}:X^{2m}\rightarrow X$ is continuous. Indeed, for any
 $\mathbf{x},\mathbf{z}\in X^{2m}$
verifying $\mathbf{\textbf{d}}_{2m}(\mathbf{x},\mathbf{z})\leq \delta$
we have
$$
\Big|\mathbb{A}(\mathbf{x})-\mathbb{A}(\mathbf{z})\Big|\leq
\int_{0}^{T} \mathcal{G}(t,s){\sum_{i=1}^{2m}\Big|f_{i}(s,x_{i}(s))-f_{i}(s,z_{i}(s))\Big|}ds
$$
$$
\leq \gamma{\mathbf{\textbf{d}}_{2m}(\mathbf{x},\mathbf{z})},
$$
due to the assumption $(\mathbf{d}),$
where $\gamma=2m\eta\underset{1\leq t\leq T}{\max}\int^{T}_{1}\mathcal{G}(t,s)ds.$
 Hence $\mathbb{A}$ is continuous.
  We have shown that  the operator $\mathbb{A}$ satisfies the conditions $(i)-(iv)$
 of  Theorem \ref{roldan2}.
 It implies that $\mathbb{A}$ has a
 $\Upsilon$-fixed
  point $\mathbf{x}^{*}=(x^{*}_{1},x^{*}_{2},\ldots,x^{*}_{2m}).$ That is
\begin{eqnarray*}
  \mathbb{A}(x^{*}_{1},x^{*}_{2},x^{*}_{3},\ldots,x^{*}_{2m}) &=& x^{*}_{1},\\
  \mathbb{A}(x^{*}_{2},x^{*}_{3},\ldots,x^{*}_{2m},x^{*}_{1}) &=& x^{*}_{2}, \\
\vdots \\
 \mathbb{A}(x^{*}_{2m},x^{*}_{1},\ldots,x^{*}_{2m-1}) &=& x^{*}_{2m}.
\end{eqnarray*}
It is obvious, for any $\mathbf{x}=(x_{1},x_{2},\ldots,x_{2m}),$
$\mathbf{y}=(y_{1},y_{2},\ldots,y_{2m})\in X^{2m}$
there exists a $\mathbf{q}=(q_{1},q_{2},\ldots,q_{2m})\in X^{2m}$
 such that  $\mathbf{x}\preceq_{2m}\mathbf{q}$ and $\mathbf{y}\preceq_{2m}\mathbf{q}.$
Indeed, consider the functions $q_{i}:[1,T]\rightarrow \mathbb{R}$
\[
q_{i}(s)=\max\{x_{i}(s),y_{i}(s)\}, \,s\in [1,T].
\]
Since $x_{i}(s)$ and $y_{i}(s)$ are continuous on $[1,T],$
the functions $q_{i}(s)$ are continuous on $[1,T]$ and $x_{i}(s)\leq q_{i}(s),y_{i}(s)\leq q_{i}(s)$
 for all $1\leq i\leq 2m.$
 Therefore $\mathbb{A}$ has a unique $\Upsilon$-fixed point
 $x^{*}=(x^{*}_{1},x^{*}_{2},\ldots,x^{*}_{2m}).$
 Next we show
$$
x^{*}_{1}=x^{*}_{2}=\ldots=x^{*}_{2m}.
$$
If $x^{*}=(x^{*}_{1},x^{*}_{2},\ldots,x^{*}_{2m})$ is the $\Upsilon$-fixed point of $\mathbb{A},$
 then $y^{*}=(y^{*}_{1},y^{*}_{2},\ldots,y^{*}_{2m})$  is also a
 $\Upsilon$-fixed point of $\mathbb{A},$ where $y^{*}_{i}=x^{*}_{i+1}$ $1\leq i\leq 2m-1$
 and $y^{*}_{2m}=x^{*}_{1}.$
 However, $\mathbb{A}$ has the unique $\Upsilon$-fixed point.
 Therefore $\mathbf{x}^{*}=\mathbf{y}^{*}$ hence
$$
x^{*}_{1}=x^{*}_{2}=\ldots=x^{*}_{2m}.
$$
 Finally, we have shown that there exists a continuous function $x^{*}(t)$ such that
$$
x^{*}(t)=\mathbb{A}(x^{*},x^{*},\ldots,x^{*})(t)=\int_{1}^{T} \mathcal{G}(t,s)\Big[{\sum_{i=1}^{2m} f_{i}(s,x^{*}(s))}ds\Big]+p(t).
$$
This proves Theorem \ref{eq41}.
\end{proof}
\section{Illustrative example}
In this section, we provide a representative example to illustrate how
Theorem \ref{eq41} can be applied in solving nonlinear Hammerstein  integral equation.
Let $T>1.$ Consider the following class of nonlinear Hammerstein  integral equations.
\begin{equation}\label{ex1}
  x(t)=\frac{1}{2\ln T}\int_{1}^{T}\frac{1}{ts}\ln\Big(\frac{s+x(s)}{sx(s)}\Big)ds
  +\alpha t -\frac{1}{2}\ln \frac{1+\alpha}{\alpha\sqrt{T}}\cdot\frac{1}{t},
  \,\,\,\, \text{where} \,\,\,\, \alpha>1.
\end{equation}
\begin{thm}\label{ex11}
For every $\alpha>1,$ the equation (\ref{ex1}) has a unique solution in   $C([1, T]).$
\end{thm}
\begin{proof} Denote
\begin{align*}
f_{1}(s,t)&=\ln(s+t), &  f_{2}(s&,t)=-(\ln s+\ln t),\\
\mathcal{G}(t,s)&=\frac{1}{2\ln T}\cdot \frac{1}{ts}, & p(t)&=\alpha t -\frac{1}{2}\ln \frac{1+\alpha}{\alpha\sqrt{T}}\cdot\frac{1}{t},
\end{align*}
where $s\in [1,T]$ and $t\geq 1.$
 It is easy to see that the equation (\ref{ex1}) can be presented as the
 equation  (\ref{eq41}) by using these notations.
 Our next goal is to show that the equation (\ref{ex1}) satisfies assumptions $(\mathbf{a})-(\mathbf{e}).$
 One can easily see that the functions $f_{1},$ $f_{2},$ $\mathcal{G}$ and $p$
 are continuous. We show that the assumption  $(\mathbf{d})$ is satisfied.
 A simple calculation shows that
 $$
 2\max_{1\leq t\leq T}\int_{1}^{T}\mathcal{G}(t,s)=
 2\max_{1\leq t\leq T}\frac{1}{2\ln T}\int_{1}^{T} \frac{ds}{ts}=1.
 $$
 Let $y\geq x\geq 1.$ One can see that
 $$
 0\leq f_{1}(s,y)-f_{1}(s,x)=\ln (s+y)-\ln(s+x)=\ln\Big(1+\frac{y-x}{s+x}\Big)\leq
 \ln(1+y-x);
 $$
 $$
 - \ln(1+y-x)\leq -\ln\Big(1+\frac{y-x}{x}\Big)=-\Big(\ln y-\ln x\Big)= f_{2}(s,y)-f_{2}(s,x)\leq 0.
  $$
We next show that the assumption  $(\mathbf{e})$ is fulfilled
with $y^{0}_{1}(t)=\alpha t/2$ and $y^{0}_{2}(t)=3\alpha t/2.$
It is easily seen that
\begin{align}\label{h1212}
  H_{1}(t)&=\frac{1}{2\ln T}\int_{1}^{T} \frac{1}{ts}\cdot \ln \Big(\frac{2+\alpha}{3\alpha s}\Big)ds
  +\alpha t -\frac{1}{2}\ln \frac{1+\alpha}{\alpha\sqrt{T}}\cdot\frac{1}{t},\\ \nonumber
  H_{2}(t)&=\frac{1}{2\ln T}\int_{1}^{T} \frac{1}{ts}\cdot \ln \Big(\frac{2+3\alpha}{\alpha s}\Big)ds
  +\alpha t -\frac{1}{2}\ln \frac{1+\alpha}{\alpha\sqrt{T}}\cdot\frac{1}{t}.
\end{align}
Evaluating the integrals in (\ref{h1212}) the functions $H_{1}$ and $H_{2}$ can be simplified as follow.
\begin{align}\label{h1212a}
  H_{1}(t)&=\alpha t+ \frac{1}{2t}\cdot \ln \Big(\frac{2+\alpha}{3(1+\alpha)}\Big),\\ \nonumber
  H_{2}(t)&=\alpha t+ \frac{1}{2t}\cdot \ln \Big(\frac{2+3\alpha}{1+\alpha}\Big).
\end{align}
The task is now to show
\begin{equation}\label{h1213a00}
y^{0}_{1}(t)\leq H_{1}(t) \,\,\,\,\text{and}\,\,\,\, H_{2}(t)\leq y^{0}_{2}(t).
\end{equation}
For this,  we first show that
\begin{equation}\label{h1213}
\frac{2+3\alpha}{1+\alpha}\leq e^{\alpha}\,\,\,\,\,\,\text{for}\,\,\,\,\,\, \alpha>1.
\end{equation}
Let
$$
k(\alpha):=e^{\alpha}-\frac{2+3\alpha}{1+\alpha}.
$$
One can check that
$$
k'(\alpha)=e^{\alpha}-\frac{1}{(1+\alpha)^2}>0
$$
since $\alpha$ is positive. It implies that
$k(\alpha)$ is increasing and, in consequence,
$k(\alpha)\geq k(1)=e-2.5>0.$
From inequality (\ref{h1213}) it follows that
 \begin{equation}\label{h1213a}
\ln\Big(\frac{2+3\alpha}{1+\alpha}\Big)\leq \alpha t^{2}
\end{equation}
since $t\geq 1.$ Dividing by $2t$ and adding $\alpha t$
to the both side of (\ref{h1213a}) yields
\begin{equation}\label{h1213b}
H_{2}(t)=\alpha t+ \frac{1}{2t}\cdot \ln \Big(\frac{2+3\alpha}{1+\alpha}\Big)\leq \frac{3\alpha t}{2}=y^{0}_{2}(t).
\end{equation}
Since $\alpha >1,$ it follows that
$$
\ln \Big(\frac{3+3\alpha}{2+\alpha}\Big) \leq \ln \Big(\frac{2+3\alpha}{1+\alpha}\Big).
$$
As a consequence of the last inequality and the inequality (\ref{h1213a})
we obtain
$$
\ln \Big(\frac{3+3\alpha}{2+\alpha}\Big)\leq \alpha t^{2}.
$$
Similarly as above, dividing by $(-2t)$ and adding $\alpha t$
to the both side of the last inequality we obtain
\begin{equation}\label{h1213c}
H_{1}(t)=\alpha t+ \frac{1}{2t}\cdot \ln \Big(\frac{2+\alpha}{3+3\alpha}\Big)\geq \frac{\alpha t}{2}=y^{0}_{1}(t).
\end{equation}
 Combing inequalities (\ref{h1213b}) and (\ref{h1213c})
 we can conclude that the assumption  $(\mathbf{e})$ is fulfilled
with $y^{0}_{1}(t)=\alpha t/2$ and $y^{0}_{2}(t)=3\alpha t/2.$
This finishes the proof of Theorem \ref{ex11}.
\end{proof}

\begin{rem}
The aim of Theorem \ref{ex1} is to show a strategy of applying abstract assumptions $(a)$-$(e)$
 of Theorem \ref{th4}  in some concrete examples.
 The reader can check  that  the solution of the equation (\ref{ex11}) is
 $x(t)=\alpha t.$
\end{rem}

\section{Acknowledgement}
We would like to thank the Ministry of Education of Malaysia for providing
us with the Fundamental Research Grant Scheme 
(FRGS/1/2018/STG06/UUM/02/13.
Code S/O 14192).
\bibliographystyle{amsplain,latexsym}

\end{document}